\documentclass[11pt,a4paper]{article}
\usepackage{lipsum}
\usepackage{authblk}
\usepackage{amsmath,a4}  
\usepackage[utf8]{inputenc}
\usepackage{graphicx,enumerate,color,amsfonts,amsthm,amsmath,amssymb,mathrsfs,marginnote,dsfont,textcomp,tikz,marginnote,subfigure,pdflscape,multirow,booktabs,float,enumitem,scrtime,currfile,titleps,hyphenat,algorithm,algpseudocode,bigints,diagbox,multirow,mathtools,mdframed}
\usepackage[colorlinks,
    linkcolor={red!60!black},
    citecolor={black},
    urlcolor={blue!80!black}]{hyperref}
\usepackage{tikz,pgf,pgffor,pgfplots}
\tikzset{>=latex}
\usetikzlibrary{calc,patterns,decorations,intersections,decorations.pathmorphing,fit,backgrounds,arrows,shapes,snakes,automata,petri,positioning,,graphs,graphs.standard}
\usepgfmodule{shapes,plot}

\newtheorem{theorem}{Theorem}[section]
\newtheorem{corollary}{Corollary}[theorem]
\newtheorem{lemma}[theorem]{Lemma}

\theoremstyle{definition}
\newtheorem{definition}{Definition}[section]

\theoremstyle{remark}
\newtheorem*{remark}{Remark}

\makeatletter
\newcommand{\setword}[2]{%
  \phantomsection
  #1\def\@currentlabel{\unexpanded{#1}}\label{#2}%
}
\makeatother

\newcommand{\dd}{\textnormal{d}}

\newcommand{\diag}{\textnormal{\textbf{diag}}}

\newcommand{\R}{\mathbb{R}}
\newcommand{\T}{\mathbb{T}}
\newcommand{\X}{\mathbb{X}}
\newcommand{\RM}{\mathcal{R}}
\newcommand{\M}{\mathcal{M}}
\newcommand{\SC}{\mathcal{S}}

\allowdisplaybreaks
\usepackage[top=2cm, bottom=2cm, left=2cm, right=2cm]{geometry}
\usepackage{fancyhdr}
\newpagestyle{main}{
\setheadrule{.4pt}
\sethead{}
        {}
        {\sectiontitle}
}
\renewenvironment{abstract}{%
\hfill\begin{minipage}{0.95\textwidth}
\rule{\textwidth}{1pt}}
{\par\noindent\rule{\textwidth}{1pt}\end{minipage}}
\makeatletter
\renewcommand\@maketitle{%
\hfill
\begin{minipage}{0.95\textwidth}
\vskip 2em
\let\footnote\thanks 
{\Large \bf \@title \par }
\vskip 1.5em
{\large \@author \par}
\end{minipage}
\vskip 1em \par
}
\makeatother
\begin{document}
%
\title{\nohyphens{Advances in mosquito dynamics modeling}}
\author[a,*]{Karunia Putra Wijaya}
\author[a]{Thomas G\"{o}tz}
\author[b]{Edy Soewono}
\affil[a]{\small\emph{Mathematical Institute, University of Koblenz, 56070 Koblenz, Germany}}
\affil[b]{\small\emph{Department of Mathematics, Bandung Institute of Technology, 40132 Bandung, Indonesia}}
\affil[$\ast$]{Corresponding author. Email: \href{mailto:karuniaputra@uni-koblenz.de}{karuniaputra@uni-koblenz.de}}
\maketitle
\begin{abstract}
{\textbf{Abstract}}: It is preliminarily known that \textit{Aedes} mosquitoes be very close to humans and their dwellings, also give rises to a broad spectrum of diseases: dengue, yellow fever, and chikungunya. In this paper, we explore a multi-age-class model for mosquito population secondarily classified into indoor-outdoor dynamics. We accentuate a novel design for the model in which periodicity of the affecting time-varying environmental condition is taken into account. An application of optimal control with collocated measure as apposed to widely-used prototypic smooth time-continuous measure is also considered. Using two approaches: least-square and maximum likelihood, we estimate several involving undetermined parameters. We analyze the model enforceability from the biological point of view such as the existence, uniqueness, positivity and boundedness of solution trajectory, also the existence and stability of (non)trivial periodic solution(s) by means of the basic mosquito offspring number. Some numerical tests are brought along at the rest of the paper as a compact realistic visualization of the model.
~\\
~\\
{\textbf{Keywords}}: \textsf{Mosquito dynamics modeling, multi-age-class model, indoor-outdoor dynamics, non-autonomous system, least-square, maximum likelihood, optimal control, the basic mosquito offspring number}
\vspace*{-5pt}
\end{abstract}


\section{Introduction}
We consider a mathematical model of mosquito population dynamics within the framework similar to that in \cite{WGS2013d,WGS2014d}. The corresponding population evolution is formulated as the initial value problem (IVP) of non-autonomous system
\begin{equation}\label{eq:sysbb}
\dot{ x}= f(t, x, u;\eta),\quad t\in[0,T],\, x(0)= x_0\succeq 0.
\end{equation}
In a standard setting, the system equation contains a hyperparameter $\eta$ whose appearance describes a collection of measurable intrinsic factors: natural birth rates, natural death rates, age-based transition rates, driving forces, etc. Studying the qualitative behavior of the solution, one demands the fluctuation phenomena within those factors to be distinguished as those of which essentially depend on time. As a consequence of uncertainty in the environmental condition, some elements of $\eta$ differ in time with possible trends: monotonically increasing, monotonically decreasing, oscillating, or even fluctuating with Brownian-type movement. Many references have even hypothesized that such intrinsic factors may behave with periodic streamline in many cases because of environmental changes: see, \cite{Cus1998d,HE2007d,HC1997d,IMN2007d,SLP2007d,SLY2003d}. Ironically, in a national integrated mosquito management programme, for example, \textsc{Bonds} \cite{Bon2012d} summarized that fluctuating meteorology (raindrop, wind speed, air temperature, air humidity, terrestrial radiation, etc) had been out of concern during the deployment of control devices in the field. Therefore, this costed substantial inefficiencies in mass deployment. 

This dependency of parameters in time brings the model into non-autonomous groundwork. As an extrinsic factor, the control measure $ u$ is incorporated into the system for which it plays as a system regulator towards the achievement of the general objectives: minimizing both population size and cost for the control. In other words, the following objective functional
\begin{equation}\label{eq:objbb}
J( u):=\frac{1}{2T}\int_{0}^{T}\sum_{i\in I_{ x}}\omega_{ x,i} x_i^2+\sum_{j\in I_{ u}}\omega_{ u,j} u_j^2\,\dd t
\end{equation} 
attains its minimum for given positive weighting constants $\{\omega_{ x,i}\}_{i\in I_{ x}}$ and $\{\omega_{ u,j}\}_{j\in I_{ u}}$, $I_{ x}=\{1,\cdots,5\}$ and $I_{ u}=\{1,2\}$.

In line with matching the underlying dynamical process of the solution with that from empirical measurement, we demonstrate an estimation of some undetermined parameters in the model by firstly setting them as random variables. Exploiting information from the system, one can characterize the solution as a handling function expressed in terms of such random variables. We utilize the property of Maximum Likelihood Estimation (MLE) and Least Square (LS) that is basically minimization of the difference between the handling function value-points and the measured data over all possible values of undetermined parameters in a bounded set. The next problem arises when analytical solution of the system, or factually the handling function, cannot be determined explicitly because of the complexity of the equation. For this reason, a schematic Local Linearization (LL) method provides a trade-off between numerical accuracy and computational outlay. Recently, several regimes in coping with MLE for parameter estimation within dynamical systems in epidemiology have been explored e.g. in \cite{YL2013d,GTC2007d,DS2007d}.

There have been numerous mathematical papers discussing the application of optimal control scenario to the mosquito reduction issue (see e.g. \cite{WGS2013d,WGS2014d,FMO2013d,TYE2010d,EY2005d,RBW2009d,WLS2006d} and some references therein). The authors used prototypic autonomous model utmost, encouraging us to propose a novel approach adopting non-autonomous dynamical system theory. In this paper we restrict our main scope to the application of temephos and Ultra Low Volume (ULV) aerosol. Enhancement of indoor-outdoor dynamics and utilization of polynomial collocation design to the control measure are parts of our interests. The choice of the design aims at achieving minimization of the costly objective meanwhile pronouncing more efficient and accurate control deployment.

\section{Model and analyses}
\label{chapad:anal}
In favor of the IVP \eqref{eq:sysbb}, $ x$ denotes the time-variant state that folds five consecutive elements, each represents the number of: indoor eggs $ x_1$, outdoor eggs $ x_2$, indoor larvae $ x_3$, outdoor larvae $ x_4$ and adults $ x_5$. A control measure $ u$ is injected into the system as an active feedback regulator whose elements represent the impact rates for the investment of: temephos $ u_1$ and ULV aerosol $ u_2$. Note that we omit writing the argument $t$ when it is obvious. To go into detail, the governing system \eqref{eq:sysbb} is unfolded as
\begin{subequations}
\begin{eqnarray}
\dot{ x}_1&=& p\alpha(t) x_5-\beta_1 x_1-q u_1 x_1-\mu_1 x_1,\\
\dot{ x}_2&=& (1-p)\alpha(t) x_5-\beta_2 x_2-\mu_2 x_2,\\
\dot{ x}_3&=& \beta_1 x_1 - \gamma_1 x_3^2 - \beta_3 x_3 -  u_1 x_3 - r u_2 x_3 - \mu_3 x_3,\\
\dot{ x}_4&=& \beta_2 x_2 - \gamma_2 x_4^2 - \beta_4 x_4 - s u_2 x_4 - \mu_4 x_4,\\
\dot{ x}_5&=& \beta_3 x_3 + \beta_4 x_4 -  u_2 x_5 - \mu_5 x_5.
\end{eqnarray}
\end{subequations}
All the involved parameters are positive and are briefly explained as follows: $p,q,r,s$ are the plausible probabilistic constants; $\alpha(t):=\epsilon+\epsilon_0\cos(\sigma t)$ (where $\epsilon>\epsilon_0>0$) is the birth rate of potential eggs depending qualitatively on meteorology distribution; $\beta_{\{1,2,3,4\}}$ are the age-transitional rates for the corresponding classes; $\gamma_{\{1,2\}}$ (where $\gamma_1>\gamma_2$) are the driving forces for the competition amongst larvae and; $\mu_{\{1,2,3,4,5\}}$ are the death rates for the corresponding classes. 

It is assumed that the control measure $ u$ be in $U:=\hat{C}^0([0,T];\mathcal{B})$ the set of piecewise-continuous functions where $\mathcal{B}$ is a bounded rectangle in $\R^{2}_+$. Because $ f$ is uniformly locally \textsc{Lipschitz} continuous on the state domain and is piecewise continuous on $[0,T]$, then $ x$ should lie in $\hat{C}^1([0,T];\R^5)$ the set of piecewise-differentiable functions. In this non-autonomous model, we call $ x(t)=\nu(t, x_0, u)$ the solution of \eqref{eq:sysbb} as a \emph{process}. In a specified case, when $\epsilon_0=0$, typical analyses: existence and uniqueness, positive invariance, and existence and stability of equilibria as regards the basic mosquito offspring number can be referred to our preceding work \cite{WGS2014d}.

Consider $\M$ as a non-autonomous set where $\M\subseteq [0,T]\times \R^5$. Denote by $\mathcal{M}_t:=\{ x\in \R^5:\,(t, x)\in\M\}$ $t$-fiber of $\M$ for each $t\in[0,T]$.
\begin{definition}[Positively Invariance]
An autonomous set $\mathcal{M}\subseteq [0,T]\times \R^5$ is called positively invariant under the process $\nu$ if for any $ x\in \M_0$, $\nu(t, x, u)\in \M_t$ for all $t\geq 0$.
\end{definition}
\begin{theorem}
It holds that $\M= [0,T]\times \R^5_+$ is positively invariant under the process $\nu$.
\end{theorem}
\begin{proof}
Let $n$ be $(5\times5)$-matrix representing a collection of all normal vectors (by rows) to the boundary of $\R^5_+$, $\partial \R^5_+$. It follows that $n=-\text{id}$ where $\text{id}$ is the identity matrix. Notice that at $i$th boundary, $\partial_i \R^5_+$, 
\begin{equation*}
\left.\left[n f(t, x, u)\right]_i\right|_{ x\in\partial_i \R^5_+, u\in U}\leq 0\,\text{ uniformly for all } t\geq 0,\,i\in I_{ x}.
\end{equation*}
This generates evidence showing that evolution of the solution points in $\partial \R^5_+$ is in counter-direction or at least perpendicular to the corresponding normal vectors. Thus, such trajectory of points, emanated from all $t\geq 0$, cannot leave $\R^5_+$.
\end{proof}

We next consider a decomposition over $ f$ as $ f(t, x, u)=A(t) x+c_1 x_3^2+c_2 x_4^2$ to exemplify further analysis. As per the decomposition, $A(t)$ is the \textsc{Jacobi}an of $f$ evaluated at $0$. Let $V:\R_+\times \R^5_+\times \R^5_+\rightarrow \R$ be a function defined as
\begin{equation*}
V(t, x, y):=\lVert x- y\rVert^2.
\end{equation*}
This function instantly agrees on the following three conditions: (i) $V>0$ if $ x\neq  y$ and $= 0$ if $ x= y$, (ii) $V$ is uniformly locally \textsc{Lipschitz} continuous on $\text{dom}(V)$ and (iii) for any $\{ x_n\}_{n\in \mathbb{N}},\,\{ y_n\}_{n\in \mathbb{N}}\in \R^5_+$, $\lim_{n\rightarrow \infty} V(t, x_n, y_n)= 0$ implies $\lim_{n\rightarrow \infty}\lVert x_n- y_n\rVert =0$.

\begin{lemma}\label{lem:vbb}
The following assertion applies:
\begin{equation*}
\lim_{h\rightarrow 0^-}h^{-1}(V(t+h, x+h f(t, x, u), y+h f(t, y, u))-V(t, x, y))\leq \wp(t,V(t, x, y))
\end{equation*}
where $\wp(t,w)$ is a continuous function exceeding $2\lVert A(t)\rVert w$ for all $t,w\in\R_+$.
\end{lemma}
\begin{proof}
The inequality follows directly from straight-forward computation (ref. \textsc{Cauchy--Schwarz} inequality).
\end{proof} 

\begin{lemma}\label{lem:gbb}
The following scalar non-autonomous equation
\begin{equation*}
\dot{w}=\wp(t,w)
\end{equation*}
holds these two conditions: (i) $\wp(t,0)=0$ for all $t\in\R_+$ and (ii) for each $\tau\in(0,+\infty)$, $w\equiv 0$ is the only solution on $[0,\tau]$ satisfying $w(0)=0$.
\end{lemma}
\begin{remark}
In our model, $A$ is a matrix-valued function over $t$ whose elements contain the continuous function $\alpha$ and piecewise continuous function $ u$. Along with the definition, it is clear that both $\alpha$ and $ u$ be bounded by some continuous function, that is, there exists a continuous function $f$ such that $\lVert u(t)\rVert\leq f(t)$ for all $t\in\R_+$. Thus, without lost of generality, the $1$--norm $\lVert A(t)\rVert_1$ infers an easiness in proving that any induced norm $\lVert A(t)\rVert$ is bounded for all $t\in\R_+$. This can immediately be seen by assigning the matrix norm equivalence, in the sense that there exists a bounded positive function $C(t)$ such that $\lVert A(t)\rVert\leq C(t)\lVert A(t)\rVert_1$ for all $t$ and any induced norm $\lVert\cdot\rVert$. Now, the existence of such function $\wp$ in the form $\wp(t,w)=\delta(t)w$ can be drawn upon the fact that $\lVert A(t)\rVert$ is bounded for all $t\in\R_+$. 
\end{remark}

We further recognize that Lemmas~\ref{lem:vbb} and~\ref{lem:gbb} are inherent substances to prove the existence of a unique solution of \eqref{eq:sysbb}. The notion of uniqueness in non-autonomous dynamical system dates back to the seminal works by \textsc{Murakami} \cite{Mur1966d}, \textsc{Ricciardi--Tubaro} \cite{RT1973d} and \textsc{Kato} \cite{Kat1976d}, where all conditions stipulated in Lemmas~\ref{lem:vbb} and~\ref{lem:gbb} set the conforming requirements. Adopting materials from \cite{Kat1976d,Kat1975d,Web1972d}, we summarize the existence and uniqueness of the solution through the following corollary.
\begin{corollary}[Uniqueness]\label{cor:1bb}
Initial value problem \eqref{eq:sysbb} has a unique solution $ x$ defined on $\R_+$ that maps to $\R^5_+$ for every $ x_0\in \R^5_+$.
\end{corollary}
To see boundedness of the existing solution, we preliminarily impose the following assumptions on \eqref{eq:sysbb}:
\begin{enumerate}[label=$(H\arabic*)$]
\item the hyperparameter $\eta$ is chosen to lie in the following set
\begin{equation*}
\left\{\eta\in \R^{18}_+:\, x_3(t;\eta)\geq  x_5(t;\eta),\, x_4(t;\eta)\geq  x_5(t;\eta)\,\text{for all }t\in\R_+\text{ and }\lVert\eta\rVert<\infty\right\},
\end{equation*}
\label{it:H1dd}
\item there exist positive continuous functional $\vartheta$ and sufficiently large constant $L$ such that $\eta$ lies in
\begin{equation*}
\left\{\eta\in \R^{18}_+:\,\frac{c(t)+\vartheta(t)}{2(\gamma_1+\gamma_2)}\lVert x(t;\eta)\rVert^2\leq L+ x_5(t;\eta)^3\,\text{for all } t\in\R_+\right\}
\end{equation*}
where $c(t):=\max\{p\alpha(t)-\beta_1-2\mu_1,(1-p)\alpha(t)-\beta_2-2\mu_2,\beta_1-\beta_3-2\mu_3,\beta_2-\beta_4-2\mu_4,\alpha(t)+\beta_3+\beta_4-2\mu_5\}$.\label{it:H2dd}
\end{enumerate}
\begin{theorem}[Lyapunov Method]\label{thm:lyabb}
Consider the IVP \eqref{eq:sysbb}. If there exist a \textsc{Lyapunov} function $\Psi:\R_+\times \R^5_+\rightarrow \R$, positive integers $\varsigma_{\{1,2,3\}}$, positive and nonnegative real numbers $L,\zeta$, and positive continuous functionals $\vartheta_{\{1,2,3\}}$ where $\vartheta_1$ is nondecreasing satisfying the following conditions:
\begin{enumerate}[label=$(A\arabic*)$]
\item $\vartheta_1\lVert x\rVert^{\varsigma_1}\leq \Psi(t, x)\leq \vartheta_2\lVert x\rVert^{\varsigma_2}$
\item $\dot{\Psi}(t, x)\leq -\vartheta_3\lVert x\rVert^{\varsigma_3}+L$
\item $\Psi(t, x)-\Psi(t, x)^{\varsigma_3\slash\varsigma_2}\leq \zeta$
\end{enumerate}
for all $t\in\R_+$, then $ x$ is bounded.
\end{theorem}
\begin{proof}
We aim at computing the following derivative $\frac{\dd}{\dd t}\Psi(t, x)\exp(at)$ where $a$ is a positive constant needed to be determined later. It is clear that 
\begin{align*}
\frac{\dd}{\dd t}\Psi(t, x)\exp(at)&=\left(\dot{\Psi}(t, x)+a \Psi(t, x)\right)\exp(at)\stackrel{\text{(A2)}}{\leq}\left(-\vartheta_3\lVert x\rVert^{\varsigma_3}+L+a \Psi(t, x)\right)\exp(at)\\
&\stackrel{\text{(A1)}}{\leq} \left(-\frac{\vartheta_3}{\vartheta_2^{\varsigma_3\slash\varsigma_2}}\Psi(t, x)^{\varsigma_3\slash\varsigma_2}+L+a \Psi(t, x)\right)\exp(at)\stackrel{\text{(A3)}}{\leq}(L+a\zeta)\exp(at)
\end{align*}
by taking $a:=\inf_{t\in \R_+}\frac{\vartheta_3}{\vartheta_2^{\varsigma_3\slash\varsigma_2}}$. This results in $\Psi(t, x)=\left(\Psi(0, x_0)+\frac{(L+a\zeta)}{a}\exp(at)-\frac{(L+a\zeta)}{a}\right)\exp(-at)\stackrel{\text{(A1)}}{\leq} \vartheta_2(0)\lVert x_0\rVert^{\varsigma_2}\exp(-at)+\frac{(L+a\zeta)}{a}\leq \vartheta_2(0)\lVert x_0\rVert^{\varsigma_2}+\frac{(L+a\zeta)}{a}$. Using the left-side inequality in (A1), we obtain $\lVert x\rVert\leq \vartheta_1(0)^{-1\slash\varsigma_1}\left(\vartheta_2(0)\lVert x_0\rVert^{\varsigma_2}+\frac{(L+a\zeta)}{a}\right)^{1\slash\varsigma_1}$ since $\vartheta_1$ is nondecreasing.
\end{proof}

\begin{corollary}
If \ref{it:H1dd} and \ref{it:H2dd} are satisfied, then any solution of \eqref{eq:sysbb} is bounded on $\R_+$.
\end{corollary}

\begin{proof}
This is a direct consequence of applying Theorem~\ref{thm:lyabb} to \eqref{eq:sysbb} with $\varsigma_{\{1,2,3\}}=2$, $\vartheta_{\{1,2\}}\equiv 1$, $\vartheta_3=\vartheta$, $\zeta=0$, and $\Psi(t, x)=\lVert x\rVert^2$.
\end{proof}

\begin{remark}
Now we have an appropriate bound for the solution, namely $\left(\lVert x_0\rVert^{2}+L\right)^{\frac{1}{2}}$.
\end{remark}

\begin{definition}[Fundamental Matrix]
Consider a non-autonomous linear system
\begin{equation}\label{eq:zzbb}
\dot{z}=W(t)z,\quad t\in\R_+,\,z(0)=z_0,
\end{equation}
with $Z(t,0)$ as the fundamental matrix for $W(t)$, that is, $z(t,z_0)=Z(t,0)z_0$, satisfying $Z(0,0)=\text{id}$. It can further be verified that (i) $Z(t,s)Z(s,0)=Z(t,0)$, (ii) $Z(t,0)^{-1}=Z(0,t)$, and (iii) $Z(t,t)=\text{id}$.
\end{definition}
\begin{theorem}\label{thm:gammabb}
Consider \eqref{eq:zzbb} where $W(t)$ is $2\pi\slash\sigma$-periodic. There exist a differentiable $2\pi\slash\sigma$-periodic matrix $\Gamma_1(t)$ and a constant matrix $\Gamma_2$ such that the according fundamental matrix
\begin{equation*}
Z(t,0)=\Gamma_1(t)\exp(\Gamma_2 t).
\end{equation*}
\end{theorem}
\begin{remark}
We shall present several highlighting insights regarding Theorem~\ref{thm:gammabb}:
\begin{itemize}
\item $\Gamma_1(t),\Gamma_2$ need not be unique and real even though $W(t)$ is real.
\item The theorem holds for $W(t)$ complex.
\item Because $Z(0,0)=\text{id}$, or eventually $\Gamma_1(0)=\Gamma_1(2\pi\slash\sigma)=\text{id}$, then $Z(2\pi\slash\sigma,0)=\Gamma_1(2\pi\slash\sigma)\exp(\Gamma_2 2\pi\slash\sigma)=\exp(\Gamma_2 2\pi\slash\sigma)$.
\end{itemize}
\end{remark}
\begin{definition}[Floquet Exponent and Floquet Multiplier]
Let $\rho(A)$ be the spectrum of $A$. In Theorem~\ref{thm:gammabb}, each element of $\rho(\Gamma_2)$ is called as a Floquet exponent meanwhile each element of $\rho(Z(2\pi\slash\sigma,0))$ is called as a Floquet multiplier.
\end{definition}
Take a look back at \eqref{eq:sysbb}. Irrespective to the appearance of the control, $ f(t, x,0)$ can be identified and decomposed as $ f(t, x,\epsilon_0)= f_0( x)+\epsilon_0 f_1(t, x)$ where $ f_0$ is autonomous and $ f_1(t, x)=\int_{0}^{1}\frac{\partial}{\partial\epsilon_0} f(t, x,\xi\epsilon_0)\,\dd \xi$. It is preliminarily known the following: (i) $ f(t+2\pi\slash\sigma, x,\epsilon_0)= f(t, x,\epsilon_0)$ for all $t\in\R_+$, and (ii) if $ Q $ is an equilibrium point of $\dot{ x}= f_0( x)$, then $ f_0( Q )=0$. 

\begin{theorem}[Existence of Periodic Solution]
Let $ f(t, x,\epsilon_0)= f_0( x)+\epsilon_0 f_1(t, x)$ and $ Q $ be an equilibrium point of $\dot{ x}= f_0( x)$. If $2\pi\slash\sigma\rho(\nabla_{ x} f_0( Q ))\ni 2\pi\slash\sigma\lambda \notin 2\pi i\mathbb{Z}$, then there exist a neighborhood $\mathcal{U}( Q )$ and $\epsilon_1$ such that for every $|\epsilon_0|<\epsilon_1$, there exists a $2\pi\slash\sigma$-periodic solution of \eqref{eq:sysbb} with a unique initial value $ x_0= x_0(\epsilon_0)\in\mathcal{U}( Q )$.
\end{theorem}

\begin{proof}
Let $ x(t, x_0,\epsilon_0)$ be such solution. We let $ x_0\in\mathcal{U}_1( Q )\subset\R_+^{5}$ such that $ x$ exists and is unique for all $t\in\R_+$ by Corollary~\ref{cor:1bb} without losing generality for $|\epsilon_0|<\epsilon_{1,1}$. Let $\zeta(t):=\nabla_{ x_0} x(t, Q ,0)$. Because $\nabla_{ x_0} x(t, x_0,\epsilon_0)$ deductively satisfies the variational equation
\begin{equation*}
\frac{\dd}{\dd t}\nabla_{ x_0} x(t, x_0,\epsilon_0)=\nabla_{ x_0} f(t, x(t, x_0,\epsilon_0),\epsilon_0)=\nabla_{ x} f(t, x(t, x_0,\epsilon_0),\epsilon_0)\nabla_{ x_0} x(t, x_0,\epsilon_0),
\end{equation*}
with $\nabla_{ x_0} x(0, x_0,\epsilon_0)=\text{id}$, then $\zeta$ immediately satisfies
\begin{equation*}
\frac{\dd}{\dd t}\zeta=\nabla_{ x} f_0( Q )\zeta\text{ with }\zeta(0)=\text{id}.
\end{equation*}
Let $S( x_0,\epsilon_0):= x(2\pi\slash\sigma, x_0,\epsilon_0)- x_0$. Since $ f\in C^1(\R_+\times\R^5_+\times(-\epsilon_{1,1},\epsilon_{1,1});\R^5)$, then $S\in C^1(\R^5_+\times(-\epsilon_{1,1},\epsilon_{1,1});\R^5)$. It is clear that $S( Q ,0)=0$ and $\nabla_{ x_0}S( Q ,0)=\exp(\nabla_{ x} f_0( Q )2\pi\slash\sigma)-\text{id}$. Let $ v$ be an eigenvector of $\nabla_{ x} f_0( Q )$ associated with an eigenvalue $\lambda$. We know that $(\exp(\nabla_{ x} f_0( Q )2\pi\slash\sigma)-\text{id}) v=(\exp(2\pi\slash\sigma\lambda)-1) v$ making a clearance that $\exp(2\pi\slash\sigma\lambda)-1$ is an eigenvalue of $\nabla_{ x_0}S( Q ,0)$. If $2\pi\slash\sigma\lambda \notin 2\pi i\mathbb{Z}$ then $\det(\nabla_{ x_0}S( Q ,0))=\prod_{i=1}^5\exp(2\pi\slash\sigma\lambda_i)-1\neq 0$, making $\nabla_{ x_0}S( Q ,0)$ invertible. By the Implicit Function Theorem, there exist a domain $\mathcal{U}_2( Q )\times(-\epsilon_{1,2},\epsilon_{1,2})$ and a smooth $x_0(\epsilon_0)$ for $(\epsilon_0, x_0(\epsilon_0))$ defined on this domain such that $S( x_0(\epsilon_0),\epsilon_0)=0$ or eventually $ x(2\pi\slash\sigma, x_0(\epsilon_0),\epsilon_0)= x_0(\epsilon_0)$. Since $ f$ is $2\pi\slash\sigma$-periodic over $t$, then $ x(t+2\pi\slash\sigma, x_0(\epsilon_0),\epsilon_0)= x(t, x_0(\epsilon_0),\epsilon_0)$ if and only if $ x(2\pi\slash\sigma, x_0(\epsilon_0),\epsilon_0)= x_0(\epsilon_0)$. Now, letting $\mathcal{U}( Q ):=\mathcal{U}_1( Q )\cap\mathcal{U}_2( Q )$ and $\epsilon_1$ such that $(-\epsilon_1,\epsilon_1)\subset (-\epsilon_{1,1},\epsilon_{1,1})\cap(-\epsilon_{1,2},\epsilon_{1,2})$ follows the desired domain of existence.
\end{proof}

Let $b_1= \epsilon p$, $b_2=\epsilon(1-p)$, $b_3=\beta_1$, $b_4=\beta_2$, $b_5=\beta_3$, $b_6=\beta_4$ and $d_1=\beta_1+\mu_1$, $d_2=\beta_2+\mu_2$, $d_3=\beta_3+\mu_3$, $d_4=\beta_4+\mu_4$, $d_5=\mu_5$. With the same technical arrangement using the next generation method \cite{DW2002d} as in \cite{WGS2014d}, we define
\begin{equation}
\RM(d_3,d_4):=\sqrt[3]{\frac{b_1}{d_5}\frac{b_3}{d_1}\frac{b_5}{d_3}+\frac{b_2}{d_5}\frac{b_4}{d_2}\frac{b_6}{d_4}}
\end{equation}
as the so-called \textit{basic mosquito offspring number}. In the domain of interest $\R_+^5$, it has been proved in \cite{WGS2014d} that two equilibria of $\dot{ x}= f_0( x)$ exist: zero equilibrium and a positive equilibrium $ Q $.
\begin{lemma}\label{lem:trivialbb}
There exist two $2\pi\slash\sigma$-periodic solutions of \eqref{eq:sysbb} in $\R_+$: the trivial solution $ x\equiv 0$ if $\RM(d_3,d_4)\neq 1$ and nontrivial solution $ x=\nu$ associated with the nontrivial autonomous equilibrium $ Q $ if $\RM(d_3+2\gamma_1 x_3^{ \ast },d_4+2\gamma_2 x_4^{ \ast })\neq 1$ where $(\cdot,\cdot, x_3^{ \ast }, x_4^{ \ast },\cdot)= Q $.
\end{lemma}

\begin{definition}[Stable Periodic Solution]
A periodic solution $\nu$ is called stable if for every $\epsilon>0$, there exists $\delta>0$ such that $\lVert x_0-\nu(\tau)\rVert<\delta$ implies $\lVert x(t,\tau, x_0)-\nu(t)\rVert<\epsilon$ for all $t\geq \tau$. Additionally, if $\lim_{t\rightarrow\infty}\lVert x(t,\tau, x_0)-\nu(t)\rVert=0$ then $\nu$ is called asymptotically stable.
\end{definition}
\begin{theorem}[Stability of Periodic Solution]
Let $ y:= x-\nu$ and irrespective to the appearance of the control, $ h(t, y) :=  f(t, y+\nu)- f(t,\nu)-W(t) y$ where $W(t):=\nabla_{ x} f(t,\nu)$. If the following conditions hold:
\begin{enumerate}[label=\textnormal(B\arabic*)]
\item $\lVert h(t, y)\rVert\leq K\lVert y(t)\rVert^2$ for some constant $K$ and all $t\in \R_+$
\item all the Floquet exponents of $\nu$ that correspond to $W(t)$ lie in the open left-half plane in $\mathbb{C}$ (the Floquet multipliers lie in the open unit disk in $\mathbb{C}$)
\item in the decomposition $\exp(\int_{0}^tW(s)\,\dd s)=\Gamma_1(t)\exp(\Gamma_2 t)$, $\Gamma_1(t)$ is bounded w.r.t. $\lVert\cdot\rVert$ for all $t\in\R_+$,
\end{enumerate}
then $\nu$ is asymptotically stable.
\end{theorem}
\begin{proof}
It follows that $ y$ satisfies
\begin{equation*}
\dot{ y}=W(t) y+ h(t, y),\quad
\end{equation*}
where $ h(t,0)=0$ and $\nabla_{ y} h(t,0)=0$.

Observe that $\lVert Z(t,0)\rVert=\left\lVert\exp(\int_{0}^tW(s)\,\dd s)\right\rVert\leq \lVert(\Gamma_1(t)\rVert\left\lVert\exp(\Gamma_2 t)\right\rVert\stackrel{\text{(B2)}}{\leq}\lVert\Gamma_1(t)\rVert C_0\exp(-\lambda t)\stackrel{\text{(B3)}}{\leq} C\exp(-\lambda t)$ for some positive constants $C$ and $\lambda$. Fix $\delta$ where $CK\delta-\lambda<0$. Define $0<a :=\lambda-CK\delta$ and $b:=\frac{\delta}{C}$. If $\lVert  y_0\rVert\leq b$, then by continuity of the vector field, there exists $\tau$ such that $ y$ exists on $[0,\tau]$ satisfying $\lVert y(t)\rVert\leq \delta$ for all $t\in [0,\tau]$. Generating the solution, we obtain $\lVert y\rVert=\left\lVert Z(t,0) y_0+Z(t,0)\int_0^tZ(s,0)^{-1} h(s, y)\,\dd s\right\rVert\stackrel{\text{(B1)}}{\leq} [\lVert y_0\rVert C+CK\delta\int_0^t\exp(\lambda s)\lVert  y(s)\rVert\,\dd s]\exp(-\lambda t)$. Employing the \textsc{Gronwall}'s Lemma, we obtain $\exp(\lambda t)\lVert  y(t)\rVert\leq \lVert y_0\rVert C\exp(CK\delta t)$ or $\lVert y(t)\rVert\leq \lVert y_0\rVert C\exp(-a t)\leq\delta \exp(-a t)$ such that $\lim_{t\rightarrow\infty}\lVert y(t)\rVert=0$. Using some extension theorem, it can be proven that there exists $\epsilon>0$ such that $ y$ is defined on $[0,\tau+\epsilon)$ by continuity of $\lVert y\rVert$. For all points, without lost of generality, $\{t_n\}_{n\in\mathbb{N}}:=\{\tau+(1-1\slash n)\epsilon\}_{n\in\mathbb{N}}$, note that $\lVert y(t_n)\rVert\leq \delta \exp(-a t_n)<\delta$ for $\lVert y_0\rVert<b$. This contradicts maximality of $\tau$.
\end{proof}
\begin{lemma}\label{lem:zerobb}
The trivial periodic solution $\nu\equiv 0$ of \eqref{eq:sysbb} is asymptotically stable whenever $\RM(d_3,d_4)<1$.
\end{lemma}
\begin{proof}
Using direct computation, it can be shown that
\begin{align}
&K>0\text{ exists},\,\Gamma_1(t)=\left[ \begin {array}{ccccc} 
1&0&0&0&\frac{p\epsilon_0}{\sigma}\sin(\sigma t)\\
0&1&0&0&\frac{(1-p)\epsilon_0}{\sigma}\sin(\sigma t)\\
0&0&1&0&0\\
0&0&0&1&0\\
0&0&0&0&1
\end {array} \right]\text{ bounded, and}\nonumber\\
&\Gamma_2=\left[ \begin {array}{ccccc} 
-d_1&0&0&0&b_1\\
0&-d_2&0&0&b_2\\
b_3&0&-d_3&0&0\\
0&b_4&0&-d_4&0\\
0&0&b_5&b_6&-d_5
\end {array} \right]\label{eq:gamma2}
\end{align}
resulting $\Gamma_1(2\pi\slash\sigma)=\text{id}$. Noticing \cite{WGS2014d} Theorem 3.2, it has been proven whenever $\RM(d_3,d_4)<1$ then $\rho(\Gamma_2)$ lies in the open left-half plane in $\mathbb{C}$.
\end{proof}
\begin{theorem}\label{thm:nontrivialbb}
Let $\nu$ be a nontrivial periodic solution, $\bar{\nu}_{\{3,4\}}(t):=\frac{1}{t}\int_{0}^{t}\nu_{\{3,4\}}(s)\,\dd s$ where $\nu_{\{3,4\}}(0)=\nu_{\{30,40\}}$ and $\nu^{\min}_{\{3,4\}}:=\min_{t\in[0,2\pi\slash\sigma]}\bar{\nu}_{\{3,4\}}(t)$. If $\RM(d_3+2\gamma_1\nu^{\min}_3,d_4+2\gamma_2\nu^{\min}_4)<1$, then $\nu$ is asymptotically stable.
\end{theorem}
\begin{proof}
Let $ y= x-\nu$. It is clear that $ y$ follows $\dot{ y}= g(t, y)$ where $ g(t, y)=A(t) y+c_1( y_3^2+2\nu_3 y_3)+c_2( y_4^2+2\nu_4 y_4)$ by recalling our decomposition upon $ f$ in \eqref{eq:sysbb}. Now, we can state that $ y$ delineates non-autonomous system with perturbances $\nu_{\{2,3\}}$ where $0$ is the trivial periodic solution. We can easily obtain the correspondence matrix for linearized system $W(t)=\nabla_{ x} g(t,0)$. We briefly state that the fundamental matrix $Z(t)=Z(t,0)=\exp(\int_{0}^tW(s)\,\dd s)$ cannot be presented easily as $\Delta_1(t)\exp(\Delta_2 t)$. To continue proceeding, the idea is by choosing $\Delta_1(t):=Z(t)\exp(-\Delta_2 t)$ and $\Delta_2$ for which $\exp(\Delta_2 2\pi\slash\sigma)=Z(0)^{-1}Z(2\pi\slash\sigma)$. Thus, this choice satisfies $Z(t)Z(0)^{-1}Z(2\pi\slash\sigma)=\Delta_1(t)\exp(\Delta_2 t)\cdot\exp(\Delta_22\pi\slash\sigma)=\Delta_1(t+2\pi\slash\sigma)\exp(\Delta_2(t+2\pi\slash\sigma))=Z(t+2\pi\slash\sigma)$, which is nothing else but the so-called \textsc{Floquet} theorem. Consider $\Gamma_2=\Gamma_2(d_3,d_4)$ as in \eqref{eq:gamma2}, it is clear that $Z(2\pi\slash\sigma)=\exp\left(2\pi\slash\sigma\Gamma_2(d_3+2\gamma_1\bar{\nu}_3(2\pi\slash\sigma),d_4+2\gamma_2\bar{\nu}_4(2\pi\slash\sigma))\right)$. One immediately obtains $\Delta_2=\Gamma_2(d_3+2\gamma_1\bar{\nu}_3(2\pi\slash\sigma),d_4+2\gamma_2\bar{\nu}_4(2\pi\slash\sigma))$. If $\RM(d_3+2\gamma_1\nu^{\min}_3,d_4+2\gamma_2\nu^{\min}_4)<1$ then $\RM(d_3+2\gamma_1\bar{\nu}_3(2\pi\slash\sigma),d_4+2\gamma_2\bar{\nu}_4(2\pi\slash\sigma))<1$, therefore with similar consequence as in Lemma~\ref{lem:zerobb}, $\rho(\Delta_2)$ lies in the open left-half plane in $\mathbb{C}$. Simultaneously, because $\RM(d_3+2\gamma_1\bar{\nu}_3(t),d_4+2\gamma_2\bar{\nu}_4(t))\leq\RM(d_3+2\gamma_1\nu^{\min}_3,d_4+2\gamma_2\nu^{\min}_4)<1$ then $Z$ is bounded, and therefore, $\Delta_1$ is bounded. The choice of $\nu^{\min}_{\{3,4\}}$ in $[0,2\pi\slash\sigma]$ returns from the fact that $\bar{\nu}_{\{3,4\}}$ have the greatest deviation on their amplitude at this range. 
\end{proof}

\section{Parameter estimation}
Let $\eta$ be the hyperparameter of the model \eqref{eq:sysbb} and $\mathcal{P}\subset\R^{18}$ be its feasible region. Let $\mathcal{I}:=\{i\in\{1,\cdots,18\}:\eta_i\text{ unfixed}\}$ and $\theta$ denote a vector which collects all associated parameters whose index is in $\mathcal{I}$. Let $\Theta\subset\mathcal{P}$ respectively be the feasible region for $\theta$. The next key enabling technical simplification is that one can further rearrange the elements of $\eta$ as $\eta=(\eta_f^{\top},\theta^{\top})^{\top}$. In order to find an estimate of $\theta$, it is essential to identify whether the system in nature is under control intervention or not. For the sake of simplicity, let us assume that there are no control treatments during the matching process. Fixing $\eta_f$ and setting $ u\equiv0$, we recast IVP \eqref{eq:sysbb} as
\begin{equation}\label{eq:st2}
\dot{ x}(t)=\bar{ f}(t, x(t);\theta),\quad t\in[0,T], x(0)= x_0,\theta\in\Theta.
\end{equation}
Let $\mathcal{J}:=\{0,\cdots,N\}$ and
\begin{equation}
\mathbb{G}_N:=\{t_j:\,t_j=j\Delta t,\,t_N=t_f,\,j\in \mathcal{J}\}
\end{equation}
be our set of discrete time-points. Taking a good solver for ODE, we assume that Eq.~\eqref{eq:st2} results in the discrete process $\phi:\mathbb{G}_N\times 0 \times \R_+^{5}\times \Theta\rightarrow \R_+^{5}$ mapping $(t_j,0, x_0,\theta)$ to $ x_j$ where the sequence $\{ x_j\}_{j\in\mathcal{J}}$ conforms the \emph{regressing path}. Let $H:\R_+^{5}\rightarrow \R^m$ be a function such that $\Phi:=H\circ \phi:\mathbb{G}_N\times 0 \times \R_+^{5}\times \Theta\rightarrow\R^m$. Assume that $ x_0$ is fixed, leading to the exposition $\Phi:\mathbb{G}_N\times\Theta\rightarrow\R^m$.

Let $\mathcal{K}\subseteq I_{ x}$. Suppose that it is given a dataset $\{\hat{\T},\hat{\X}\}$ of time-state points which folds the sample $\{\hat{\T}^i_l,\hat{\X}^i_l\}^{i\in\mathcal{K}}_{l=1,\cdots,k_i}$ and let
\begin{equation}
\mathcal{A}:=\left\{i\in\mathcal{K}:0< \hat{\T}^i_l< T,l=1,\cdots,k_i\right\}.
\end{equation}
The first main process in the parameter estimation is given briefly as follows. In practice, since most $\hat{\T}^i_l$ is beyond $\mathbb{G}_N$, notorious interpolation and extrapolation processes are needed for all $i\in\mathcal{A}$ and only interpolation process for all $i\in\mathcal{K}\backslash\mathcal{A}$. The processes seek all corresponding state-points at all $t_j$ based on information from the known points given in the data set. Since an extrapolation process suffers from greater uncertainty, thus the higher $k_i$ for all $i\in\mathcal{A}$ will help to produce more meaningful results. Once we have the data set $\{\hat{\T},\hat{\X}\}$ extrapolated and interpolated with respect to the aforesaid procedures, one draws the refined data set, $\{\T,\X\}$, where it holds $\{\T^i_j:\,j\in\mathcal{J}\}=\mathbb{G}_N$ for all $i\in\mathcal{K}$ and therefore $\X:\mathbb{G}_N\rightarrow \R^{|\mathcal{K}|}_+$.

In contrast with the lack of details in the data, we will always need an $H$-like function, $\overline{H}$, which maps $\X$ from $\R^{|\mathcal{K}|}_+$ into $\R^m$ as a collection course. We assume that the data for all state-classes in the model are not necessarily known. Now the corresponding number $m$ should be taken to satisfy $1\leq m\leq \min\{|\mathcal{K}|,5\}$. Working with the same treatment as in the regressing path, we let $\mathcal{X}:=\overline{H}\circ \X:\mathbb{G}_N\rightarrow \R^m$, making $\mathcal{X}$ and $\Phi$ comparable. 
\subsection{Least-square approach}
Let us define the error of measurement
\begin{equation}
\epsilon_j(\theta):=\mathcal{X}_j-\Phi_j(\theta),\quad j\in\mathcal{J}.
\end{equation}
Let $\epsilon(\theta)=[\epsilon_0(\theta),\cdots\epsilon_{|\mathcal{J}|}(\theta)]$. Given an estimate for $\Theta$, now our problem reads as
\begin{equation}\label{eq:lsbb}
\text{find }\theta\in\Theta\text{ such that }J(\theta):=\lVert\epsilon(\theta)\rVert_F^2\rightarrow\min. 
\end{equation}
In this formulation, $\lVert\cdot\rVert_F$ denotes the Frobenius norm.
\subsection{Maximum likelihood approach}
Assume that $\{\epsilon_j(\theta)\}_{j\in\mathcal{J}}$ are considerably independent and identically distributed (iid) since it most commonly appears that the data are randomly distributed relative to the regressing path. Then we can assume $\{\epsilon_j(\theta)\}_{j\in\mathcal{J}}\stackrel{\text{iid}}{\sim}\mathcal{N}(0,\Sigma)$ with the corresponding probability density function (pdf)
\begin{equation}
\varphi(\epsilon_j;\theta)=\frac{1}{(2\pi)^{\frac{m}{2}}\det(\Sigma)^{\frac{1}{2}}}\exp\left[-\frac{1}{2}\Delta_j^2\right]
\end{equation}
where $\Delta_j=\sqrt{\epsilon_j(\theta)^{\top}\Sigma^{-1}\epsilon_j(\theta)}$ is the \textsc{Mahalanobis} distance from $\epsilon_j(\theta)$ to $0$ and $\det(\Sigma)$ is the determinant of $\Sigma$. The joint pdf (jpdf) for all random variables $\{\epsilon_j(\theta)\}_{j\in\mathcal{J}}$ is given by
\begin{equation}\label{eq:jpdfbb}
\varphi(\epsilon;\theta)=\prod_{j\in\mathcal{J}}\frac{1}{(2\pi)^{\frac{m}{2}}\det(\Sigma)^{\frac{1}{2}}}\exp\left[-\frac{1}{2}\Delta_j^2\right]=\frac{1}{(2\pi)^{\frac{m|\mathcal{J}|}{2}}\det(\Sigma)^{\frac{|\mathcal{J}|}{2}}}\exp\left[-\frac{1}{2}\lVert\Delta\rVert_2^2\right]
\end{equation}
where $\Delta=(\Delta_0,\cdots,\Delta_{|\mathcal{J}|})^{\top}$.
Independent from $\Sigma$, we get the fact
\begin{equation}
\lVert\epsilon(\theta)\rVert_F\rightarrow 0\text{ if and only if }\lVert\Delta\rVert_2\rightarrow 0\text{ if and only if }\varphi\rightarrow\max.
\end{equation}
So the most sensible way of finding a good $\theta$ is by maximizing $\varphi$, or by maximizing $\log\varphi$, since $\log$ is monotonically increasing. 
\begin{theorem}[Maximum Likelihood for Multivariate Normal Distribution]
Let $\{\epsilon_j(\theta)\}_{j\in\mathcal{J}}\stackrel{\text{iid}}{\sim}\mathcal{N}(0,\Sigma)$ and
\begin{equation*}
\SC(\theta):=\sum_{j\in\mathcal{J}}\epsilon_j(\theta)\epsilon_j(\theta)^{\top},\quad \hat{\Sigma}(\theta):=\frac{1}{|\mathcal{J}|}\SC(\theta),
\end{equation*}
then $\hat{\Sigma}(\theta)=\arg\max_{\Sigma}\log\varphi(\epsilon;\theta)$ where $\varphi(\epsilon;\theta)$ is as given in \eqref{eq:jpdfbb}.
\end{theorem}
\begin{remark}
The $(m\times m)$-random variable $\SC(\theta)$ is called as \textsc{Wishart} matrix which follows \textsc{Wishart} distribution $\mathcal{W}_m(|\mathcal{J}|,\Sigma)$ with parameters $m$ (dimension of the matrix), $|\mathcal{J}|$ (degree of freedom) and $\Sigma$ (positively defined covariance matrix). We can always perform the standardization $\SC(\theta)\sim \mathcal{W}_m(|\mathcal{J}|,\Sigma)\Leftrightarrow \Sigma^{-\frac{1}{2}}\SC(\theta)\Sigma^{-\frac{1}{2}}\sim \mathcal{W}_m(|\mathcal{J}|,I_m)$ where $\mathcal{W}_m(|\mathcal{J}|,I_m)$ denotes the standard \textsc{Wishart} distribution. In the special case when $m=1$, $\mathcal{W}_1(|\mathcal{J}|,1)=\chi^2_{|\mathcal{J}|}$. 
\end{remark}

The final constrained optimization problem reads as
\begin{equation}\label{eq:mlebb}
\text{find }\theta\in\Theta\text{ such that }J(\theta):=\log\varphi(\epsilon;\theta)=\log\frac{1}{(2\pi)^{\frac{m|\mathcal{J}|}{2}}\det(\hat{\Sigma}(\theta))^{\frac{|\mathcal{J}|}{2}}}\exp\left[-\frac{1}{2}\lVert\hat{\Delta}\rVert_2^2\right]\rightarrow\max. 
\end{equation}
In this case, $\hat{\Delta}=(\hat{\Delta}_0,\cdots,\hat{\Delta}_{|\mathcal{J}|})^{\top}$ and $\hat{\Delta}_j=\sqrt{\epsilon_j(\theta)^{\top}\hat{\Sigma}(\theta)^{-1}\epsilon_j(\theta)}$.
\begin{remark}
In line with the computation of an optimal solution using a derivative-use method, one has to find the so-called \textit{\textsc{Fisher}'s score} function $\mathbb{F}(\theta)$ which is nothing but the \textsc{Jacobi}an $\nabla_{\theta}\log\varphi(\epsilon;\theta)$ and (for \textsc{Newton}/quasi-\textsc{Newton}) the \textit{information matrix} $\mathbb{I}_{|\mathcal{J}|}(\theta)$ which is negative of the \textsc{Hess}ian $\nabla_{\theta}^2\log\varphi(\epsilon;\theta)$. These computations require very lengthy expressions and therefore one has to achieve very expensive evaluations. Nevertheless, heuristics should offer a trade-off in direct solving but limit their speed in convergence. For this reason, our initiatory computation uses the genetic algorithm to find an optimal solution. 

Another important aspect in this problem is that, by giving $\Theta$ from the scratch, the value of the parameter $\theta$ on each iterate seems to converge to the boundary of $\Theta$. Initiatively, in this paper, we impose fixed values for all parameters right up in the front, i.e. $\eta=(\eta_f^{\top},\theta_f^{\top})^{\top}$, and then perturb the resulting solution with a \textsc{Gaussian} noise along with the covariance matrix $\Sigma_f$. Matching the original with this perturbed model, one can perceive the process as $\theta_f$-recovery. Considering the estimate for $\Theta$, there would be 2 possible methods which can be used: \textsc{Wald} confidence and the profile likelihood confidence methods. For $|\mathcal{J}|$ very large, the variance $\text{var}(\theta_f)\sim\mathbb{I}_{|\mathcal{J}|}^{-1}(\theta_f)$ where $[\text{var}(\theta_f)]_{ik}=[\text{cov}(\theta_{f,i},\theta_{f,k})]$, $i,k=1,\cdots,|\mathcal{I}|$. Let $z_{1-\tau}$ be $(1-\tau)$-quantile of a standard normal distribution and $\diag Z$ be a vector composed by selecting out all main diagonal elements of $Z$. Choosing an appropriate $\tau$, we gain a \textsc{Wald} confidence interval running from \textsc{Wald} test: $(H_0:\theta=\theta_f\,\text{vs}\, H_1:\theta\neq\theta_f)$ as
\begin{equation}
\Theta=\left[\theta_f-z_{1-\tau}\diag\sqrt{\mathbb{I}_{|\mathcal{J}|}^{-1}(\theta_f)},\theta_f+z_{1-\tau}\diag\sqrt{\mathbb{I}_{|\mathcal{J}|}^{-1}(\theta_f)}\right].
\end{equation}
One thing we need to make sure that at this large $|\mathcal{J}|$, the numerical evaluation of the inverse information matrix should not be really expensive -- one can approach it with numerical approximation on derivatives. Another method which is more accurate than \textsc{Wald} confidence method for $|\mathcal{J}|$ small is the profile likelihood confidence method. The profile likelihood confidence interval (also called the likelihood ratio confidence interval) derives from the asymptotic Chi-square distribution of the likelihood ratio statistics. Let $l(\Theta)$ and $u(\Theta)$ denote the lower and upper bound of $\Theta$ respectively. It is known that
\begin{equation}
2\log\left(\frac{\varphi(\epsilon;\theta_f)}{\varphi(\epsilon;\theta)}\right)< \chi_{|\mathcal{J}|-1;1-\tau}^2,
\end{equation}
which essentially determines $l(\Theta)=\arg\left\{\varphi(\epsilon;\theta):\,\varphi(\epsilon;\theta)=\varphi(\epsilon;\theta_f)\exp\left[\frac{1}{2}\chi_{|\mathcal{J}|-1;1-\tau}^2\right]\right\}$. For the upper bound $u(\Theta)$, we take some arbitrary value in $(\theta_f,\infty)$. 

To counteract the solution from reaching the boundary of the given set, we estimate a small positive number $\varepsilon$ and refine the objective in both \eqref{eq:lsbb} and \eqref{eq:mlebb} using an interior point function as
\begin{equation}\label{eq:lsbb2}\tag{LS}
\text{find }\theta\in\Theta\text{ such that }J(\theta):=\lVert\epsilon(\theta)\rVert_F^2-\varepsilon \left[\log(\theta-l(\Theta))+\log(u(\Theta)-\theta)\right]\rightarrow\min
\end{equation}
and analogously
\begin{equation}\label{eq:mlebb2}\tag{MLE}
\text{find }\theta\in\Theta\text{ such that }J(\theta):=\log\varphi(\epsilon;\theta)+\varepsilon \left[\log(\theta-l(\Theta))+\log(u(\Theta)-\theta)\right]\rightarrow\max. 
\end{equation}
\end{remark}
\subsection{Evaluation of $\Phi$}
In order to evaluate $\Phi$, we adopt the property of the Local Linearization (LL) method as it persuades balance between computational outlay and convergence. Related to as in \cite{BJR1996d,Ram1999d}, the authors suggested to find the solution of
\begin{equation}\label{eq:LLM}
\dot{ x}(t;\theta)=\bar{ f}(t_j, x(t_j;\theta))+\nabla_{ x}\bar{ f}(t_j, x(t_j;\theta))( x(t;\theta)- x(t_j;\theta))
\end{equation}
on each subinterval $[t_j,t_{j+1})$ where $ x(t_0)= x_0$ and $t_j,t_{j+1}\in\mathbb{G}_N$. The solution of \eqref{eq:LLM} is given as the following recursion
\begin{eqnarray}\label{eq:exc}
 x_{j+1}(\theta)= x_j(\theta)+\left(\exp\left[\nabla_{ x}\bar{ f}_j(\theta)\Delta t\right]-\text{id}\right)\nabla_{ x}\bar{ f}_j(\theta)^{-1}\bar{ f}_j(\theta)
\end{eqnarray}
providing that $\nabla_{ x}\bar{ f}_j(\theta)$ is invertible. In this formulation, $ x_j(\theta)$, $\bar{ f}_j(\theta)$ and $\nabla_{ x}\bar{ f}_j(\theta)$ are the abbreviations for $ x(t_j;\theta)$, $\bar{ f}(t_j, x(t_j;\theta))$ and $\nabla_{ x}\bar{ f}(t_j, x(t_j;\theta))$, respectively.
\begin{lemma}
Let us devote to two solutions on the subinterval $[t_j,t_{j+1})$. Let $\phi$ be a process representing the analytic solution of \eqref{eq:st2} and $\phi_{\Delta t}$ be a process generated as the solution of \eqref{eq:st2} using the LL method. Assume that $\bar{ f}$ is uniformly \textsc{Lipschitz} continuous on all prescribed domains of its arguments. Independent from $\theta$, there exists a positive constant $C$ such that
\begin{equation*}
\lVert \phi-\phi_{\Delta t}\rVert_{L^1([t_j,t_{j+1}))}\leq C\Delta t^2\text{ uniformly for all } t_j,t_{j+1}\in\mathbb{G}_N.
\end{equation*}
\end{lemma}
\begin{remark}
This Lemma gives evidence that the smaller $\Delta t$ taken in numerical computation, the more the solution from LL method tends to the analytic solution. One can take a look at the analogous proof of this Lemma in \cite{RG1997d}. Another important problem to be tackled is how we can efficiently compute the matrix exponential in \eqref{eq:exc}. The interested reader can take a look into the \textsc{Pad\'{e}} approximation for matrix exponential, see e.g. \cite{ACF1996d,Loa1977d}. However, in this paper we omit writing the detail of this approximation.
\end{remark}
\section{Optimal control problem}
\subsection{Polynomial collocation}
Let us assume that the control measures are applied in every $n$ days. The spacing between times of application is assigned as $ h+ne$ where $ h,e\in\R^2$ and $e$ is a vector containing unities. Let $\tau^{1,k}:=I_k( h+ne)$ and $\tau^{2,k}:=I_k( h+ne)+ne$ be two discrete time-points where $I_k$ is a $(2\times2)$-diagonal matrix containing counters. Let $\ast$ and $\ast\slash$ respectively denote the MATLAB pointwise multiplication and division between two vectors. If $\diag(I_k)$ counts all elements in the set $\{0e,1e,2e,\cdots,(T-n)\ast\slash( h+ne)\}$ in a consecutive manner, then both $\tau^{1,k}$ and $\tau^{2,k}$ count some distinct numbers in $\R^2$. For the sake of simplicity, assume that $(T-n)$ is divisible by $ h+ne$ with respect to the operator $\ast\slash$. Therefore, we have a finite collection of intervals
\begin{equation}
\mathcal{T}:=\{[\tau^{1,k}_{1},\tau^{2,k}_{1})\times [\tau^{1,k}_{2},\tau^{2,k}_{2})\}_{k\in \{0,\cdots,\lVert(T-n)\ast\slash( h+ne)\rVert_{\infty}\}}\subset \R^2_+.
\end{equation}
Note that for all $k\in\{\min\left((T-n)\ast\slash( h+ne)\right),\cdots,\lVert(T-n)\ast\slash( h+ne)\rVert_{\infty}\}$, it appears either $[\tau^{1,k}_{1},\tau^{2,k}_{1})=\emptyset$ or $[\tau^{1,k}_{2},\tau^{2,k}_{2})=\emptyset$ simultaneously since one may need $ h$ containing distinct elements. A collection of the corresponding counters should have zero cardinality. 

Let $\delta^k_{\tau^1;\tau^2}(t)$ be some vector-valued function defined by
\begin{equation}
\delta^k_{\tau^1;\tau^2}(t)\stackrel{\text{def}}{\equiv}\left\{
\begin{array}{cl}
 p_k(t),& (t,t)\in[\tau^{1,k}_{1},\tau^{2,k}_{1})\times [\tau^{1,k}_{2},\tau^{2,k}_{2})\\
0,& (t,t)\in([0,T]\times [0,T])\backslash\left([\tau^{1,k}_{1},\tau^{2,k}_{1})\times [\tau^{1,k}_{2},\tau^{2,k}_{2})\right)
\end{array}\right.
\end{equation}
where $ p_k(t)$ is a vector-valued polynomial of degree being arbitrary. Let $\delta_{\tau^1;\tau^2}(t):=\left(\delta^0_{\tau^1;\tau^2}(t),\cdots,\delta^{\lVert(T-n)\ast\slash( h+ne)\rVert_{\infty}}_{\tau^1;\tau^2}(t)\right)$ and $ p(t):=\left( p_0(t),\cdots, p_{\lVert(T-n)\ast\slash( h+ne)\rVert_{\infty}}(t)\right)$ be $(2\times [\lVert(T-n)\ast\slash( h+ne)\rVert_{\infty}+1])$-vector-valued functions. Thus we design our control measure as
\begin{equation}
\label{eq:controlbb}
 c\left( t\right) :=  v\ast\delta_{\tau^1;\tau^2}(t).	
\end{equation}
In this case, $ v$ denotes a vector containing control measure values at $\mathcal{T}$ whose dimension is similar to that of $ p(t)$. Given a continuous-time control $ u(t)$. Define a $ p$-collocation $\Lambda(\cdot; p):U\rightarrow U$ such that for $ c=\Lambda( u; p)$, there is a weighting vector $ v$ satisfying $ c\left( t\right) =  v\ast\delta_{\tau^1;\tau^2}(t)$ and $\lVert  u- c\rVert\rightarrow\min$ for all $(t,t)\in\mathcal{T}$.
\subsection{Existence of optimal control}
Designate the transformation over time-state variables on the following performance
\begin{equation*}
 x_i\mapsto  y_i\text{ for all }i\in I_{ x}\text{ and }t\mapsto  y_6.
\end{equation*}
As a consequence, there exists a function $Y$ such that the non-autonomous equation \eqref{eq:sysbb} is similar to the following autonomous equation
\begin{equation}\label{eq:Ybb}
\dot{ y}=Y( y, u),\quad  y_i(0)= x_0 \text{ for all }i\in I_{ x}\text{ and } y_6(0)=0,\,t\in[0,T]
\end{equation}
where $Y_6( y, u)=1$.
Define $\mathcal{D}=\{ y:\dot{ y}=Y( y, u),\,t\in[0,T],\, y(0)= y_0,\, u\in U\}$ as the set of admissible states. Recall our objective functional $J( u)=\int_{0}^{T} j( y, u)\,\dd t$ as in \eqref{eq:objbb} along with this transformation and compose the optimal control problem as 
\begin{equation}
\label{eq:optbb}
\tag{OC}\text{find }( y, u)\in\mathcal{D}\times U\text{ such that }J( u)\rightarrow\min.
\end{equation}
The following lemma derives one appropriate material to prove existence of optimal control in \eqref{eq:optbb}.
\begin{lemma}{}\label{lm:Sbb}
The following set 
\begin{equation*}
\mathcal{S}(t, y):=\{ j( y, u)+\gamma,Y_{1}( y, u),\cdots,Y_{6}( y, u):\gamma\geq 0, u\in \mathcal{B}\},\, \mathcal{B}:=[0,a_1]\times[0,a_2],\,a=(a_1,a_2)\in\R^2_+
\end{equation*}
is convex for all $(t, y)\in[0,T]\times \R_+^{6}$.
\end{lemma}
\begin{proof}{}
Fix $(t, y)\in[0,T]\times \R_+^{6}$ as an arbitrary choice and write $ j( y, u)+\gamma=\kappa( u_1, u_2,\gamma)$. Keeping in mind that $Y=Y( y, u_1, u_2)$ and $\kappa=\kappa( u_1, u_2,\gamma)$ are continuous over $ u$ and $\gamma$. Let $( u_1, u_2)\in \mathcal{B}$. Now consider that  $\mathcal{S}(t, y)$ is a set of points $\xi\in\R^7$ where its structure can be studied as follows. For fixed $ u_2=0$ and $\gamma=0$, it is clear that $\xi_1=\kappa( u_1,0,0)$, $\xi_{\{2,4\}}=Y_{\{2,4\}}( y, u_1,0)$ and $\xi_{\{3,5,6,7\}}$ are constant. This means that such points generate a parametric curve in $\R^7$ whose projections on $\xi_1\xi_2$- and $\xi_1\xi_4$-plane are convex quadratic, meanwhile on each $\xi_1\xi_3$-, $\xi_1\xi_5$-plane and so forth are straight segments since $[0,a_1]$ is bounded. If $\gamma$ goes from $0$ to $\infty$, then this convex curve moves along $\xi_1$-axis from an initial position to infinity. At this stage, the generated 2D-hyperplane, say $\mathbb{P}_{\text{2D}}$, is clearly convex. Moreover, for constant $\xi_{\{3,5,6,7\}}$ we can identify $\mathbb{P}_{\text{2D}}$ in $\xi_1\xi_2\xi_4$-Cartessian space. If $ u_2$ goes from $0$ to $a_2$, then $\mathbb{P}_{\text{2D}}$ simultaneously moves along new axes, i.e. $\xi_5$- and $\xi_6$-axis. It is clear that the set $\mathbb{P}_{\text{3D}}:=\{(\xi_1,\xi_2,\xi_4,\xi_5)\in\R^4:\,(\xi_1,\xi_2,\xi_4)\in \mathbb{P}_{\text{2D}},\xi_5\in [Y_4( y, u_1,a_2),Y_4( y, u_1,0)]\}$ is convex, and therefore is the set $\mathbb{P}_{\text{4D}}:=\{(\xi_1,\xi_2,\xi_4,\xi_5,\xi_6)\in\R^5:\,(\xi_1,\xi_2,\xi_4,\xi_5)\in \mathbb{P}_{\text{3D}},\xi_6\in [Y_5( y, u_1,a_2),Y_5( y, u_1,0)]\}$. Then for fixed $(t, y)$, the set $\mathcal{S}(t, y)=\{\xi\in\R^7:\,(\xi_1,\xi_2,\xi_4,\xi_5,\xi_6)\in\mathbb{P}_{\text{4D}},\xi_3,\xi_7\,\text{constant}\}$ is also convex. 
\end{proof}

Now we are ready to prove the existence of optimal control for our model.
\begin{lemma}{}
There exists the only optimal pair $(\bar{ y},\bar{ u})$ for the optimal control problem \eqref{eq:optbb}.
\end{lemma}
\begin{proof}{}
We refer to the \textsc{Filippov--Cesari}'s Theorem \cite{Ces1983d} to prove the existence of optimal pair. It states whenever the following conditions hold
\begin{enumerate}
\item there exists an admissible pair,
\item the set $\mathcal{S}(t, y)$ defined in Lemma \ref{lm:Sbb} is convex for every $(t, y)\in[0,T]\times\R^6_+$,
\item $\text{image}\{ u:\, u\in U\}$ is compact,
\item there exists a number $\delta>0$ such that every solution $\ell( y)<\delta$ for all $t\in[0,T]$ and all admissible pairs $( y, u)$,
\end{enumerate}
then there exists such optimal pair. Ads 1 and 3 are trivial, meanwhile ad 2 is proved in Lemma \ref{lm:Sbb}. Clearly a well-defined vector field $Y$ in \eqref{eq:Ybb} conduces continuity of $ y$ on $[0,T]$. By the Bounded Value Theorem, one can easily show that $ y$ is bounded on the compact $t$-domain $[0,T]$. This completes the proof.
\end{proof}
\begin{definition}[Saturation]
Let $\hat{C}([0,T];\R^2_+)$ denote the set of piecewise-continuous functions mapping $[0,T]$ into $\R^2_+$. It is defined the saturation $\Upsilon:\hat{C}([0,T];\R^2_+)\rightarrow U:=\hat{C}([0,T];\mathcal{B})$ (where the block $\mathcal{B}= [0,a_1]\times[0,a_2]$) as
\begin{equation}
\Upsilon( u)=\max\left(0,\min\left(a, u\right)\right).
\end{equation}
\end{definition}
\begin{lemma}
Let $ u^{ \ast }\in \hat{C}([0,T];\R^2_+)$ be the optimal control of the broadening problem \eqref{eq:optbb} by expanding the space for admissible controls. Let $\Upsilon:\hat{C}([0,T];\R^2_+)\rightarrow U$ be a saturation over the control resulting the projected optimal control $\bar{ u}\in U$, or $\bar{ u}:=\Upsilon( u^{ \ast })$. Then $(\bar{ y},\bar{ u})\in \mathcal{D}\times U$ is the solution of the original problem \eqref{eq:optbb}.
\end{lemma}
Working with the similar technical arrangement as in \cite{WGS2014d} Section 4, we generate the necessary optimality conditions for optimality as follows.
\begin{theorem}[Necessary Conditions]
Consider the broadening optimal control problem \eqref{eq:optbb} by expanding the space of admissible controls. Let $ u^{ \ast }\in \hat{C}([0,T];\R^2_+)$ be the minimizer for $J$ and $ y^{ \ast }\in \hat{C}^{1}([0,T];\R^6_+)$ be the resulting state. There exists a dual variable $ z^{ \ast }\in \hat{C}^1([0,T];\R^6)$ such that the tuple $( y^{ \ast }, z^{ \ast }, u^{ \ast })$ satisfies the following system
\begin{subequations}
\label{eq:necc}
\begin{align}
&\dot{ y}^{ \ast }=\left.\frac{\partial\mathcal{H}}{\partial  z}\right|_{( y^{ \ast }, z^{ \ast }, u^{ \ast })},\, y^{ \ast }(0)= y_0\succeq0,\quad 
\dot{ z}^{ \ast }=-\left.\frac{\partial\mathcal{H}}{\partial y}\right|_{( y^{ \ast }, z^{ \ast }, u^{ \ast })},\\
&u^{ \ast }=\left.\arg\text{zero}\left(\frac{\partial\mathcal{H}}{\partial  u}\right)\right|_{( y^{ \ast }, z^{ \ast })}, \, z^{ \ast }(T)=0
\end{align}
\end{subequations}
for all $t\in[0,T]$. The function $\mathcal{H}( y, z, u):= j( y, u)+z'Y( y, u)$ is the \textsc{Hamilton}ian function, meanwhile all the equations in \eqref{eq:necc} are respectively the \textit{state}, \textit{adjoint}, \textit{gradient} equations and the \textit{transversality condition}.
\end{theorem}
The adjoint (with transversality condition) and gradient equations can now be unfolded as
\begin{subequations}
\label{eq:adjbb}
\begin{align}
\dot{ z}_{1} &=-\omega_{ x,1} y_1+ (\beta_{1}+ q u_{1}+ \mu_{1}) z_{1}- \beta_1 z_{3},& z_1(T)=0\\
\dot{ z}_{2} &= -\omega_{ x,2} y_2+ (\beta_{2}+
\mu_{2}) z_{2}-\beta_{2} z_{4},& z_2(T)=0\\
\dot{ z}_{3} &=-\omega_{ x,3} y_{3}+ (2\gamma_{1} y_3 +\beta_{3}+  u_{1} + r u_2+\mu_{3}) z_{3}- \beta_{3} z_{5},& z_3(T)=0\\
\dot{ z}_{4} &=-\omega_{ x,4} y_{4} +(2\gamma_{2} y_{4} +
\beta_{4}+ s u_2 + \mu_{4}) z_{4}- \beta_{4} z_{5},& z_4(T)=0\\
\dot{ z}_{5}&=-\omega_{ x,5} y_5+( u_{2}+ \mu_{5}) z_{5}- [\epsilon+\epsilon_0\cos(\sigma  y_6)]p  z_{1} - [\epsilon+\epsilon_0\cos(\sigma  y_6)] (1-p)  z_{2},& z_5(T)=0\\
\dot{ z}_{6}&=\sigma\epsilon_0\sin(\sigma y_6)p y_1 z_1+\sigma\epsilon_0\sin(\sigma y_6)(1-p) y_2 z_2,& z_6(T)=0
\end{align}
\end{subequations}
and
\begin{subequations}
\label{eq:gradbb}
\begin{eqnarray}
\omega_{ u,1} u_{1}- q y_{1} z_{1}- y_{3} z_{3}&\equiv&0\\
\omega_{ u,2} u_{2} -r y_3 z_{3} - s y_4 z_{4} -  y_5 z_{5}&\equiv&0.
\end{eqnarray}
\end{subequations}
The following algorithm illustrates our scheme to solve \eqref{eq:optbb}.
\begin{mdframed}
\begin{description}
\item[\textsc{Gradient method for solving \eqref{eq:optbb}}]\label{alg:gradddd}
\item[Return:] The tuple $(\hat{ y},\hat{ u},\hat{J})$. 
\item[Step 0] Set $k=0$, an initial guess for the control $u^{k}\in\mathcal{U}$, an error tolerance $\epsilon>0$ and an initial step-length $\lambda>0$.
\item[Step 1] Compute the $p$-collocation $ u^k\leftarrow \Lambda( u^k; p)$.
\item[Step 2] Compute $ y^k_{ u}\leftarrow y^{k}(\cdot; u^k)$ and $ z^k_{ y, u}\leftarrow z^{k}(\cdot; y^{k}(\cdot; u^k))$ consecutively from the state (forward scheme) and adjoint equation (backward scheme).
\item[Step 3] Compute the objective functional $J( u^k)$.
\item[Step 4] Compute $ u^{ \ast  k}( y^{k}_{ u}, z^{k}_{ y, u})$ from the gradient equation and set $ u^{ \ast  k}\leftarrow\Lambda( u^{ \ast  k}, p)$.
\item[Step 5] Compute $J(u^{k+1})$ and set $\Delta J\leftarrow J(u^{k+1})-J(u^{k})$.
\item[Step 6] Update $ u^{k+1}(\lambda)\leftarrow u^k+\lambda  u^{ \ast  k}$ and $ u^{k+1}\leftarrow\Upsilon( u^{k+1})$. Compute $ y^{k+1}_{ u}$ and $ z^{k+1}_{ y, u}$.
\item[Step 7] Compute $J( u^{k+1})$ and set $\Delta J\leftarrow J( u^{k+1})-J( u^{k})$.
\item[Step 8] If $|\Delta J|<\epsilon$, then set $(\hat{ y},\hat{ u},\hat{J})\leftarrow( y^{k+1}_{ u}, u^{k+1},J( u^{k+1}))$ then stop.
\item[Step 9] While $\Delta J\geq 0$ do
\begin{enumerate}[label=(9.\arabic*)]
\item Update new $\lambda\leftarrow\arg\min_{s\in[0,\lambda]}\psi(s):=J( u^{k+1}(s))$ where $\psi$ is a quadratic function as representative of $J$ with respect to the step-length $s$. 
\item Compute new $ u^{k+1}(\lambda)\leftarrow u^k+\lambda  u^{ \ast  k}$ and set $ u^{k+1}\leftarrow\Upsilon( u^{k+1})$. Then compute new $ y^{k+1}_{ u}$ and $ z^{k+1}_{ y, u}$.
\item Compute $J( u^{k+1})$ and set $\Delta J\leftarrow J( u^{k+1})-J( u^{k})$.
\item If $|\Delta J|<\epsilon$, then set $(\hat{ y},\hat{ u},\hat{J})\leftarrow( y^{k+1}_{ u}, u^{k+1},J( u^{k+1}))$ then stop.
\end{enumerate}
\item[Step 8] Set $k\leftarrow k+1$ and go to \textbf{Step 4}.
\end{description}
\end{mdframed}
\begin{remark}
From Step (9.1), we note that the solution exists since $\psi(0)$, $\psi'(0)$ and $\psi(\lambda)$ can be computed directly. A termination criterion is also included in this step namely: when $\lambda<\lambda_0$ for sufficiently small $\lambda_0$. 
\end{remark}

\section{Numerical tests}
Table \ref{tab:par} gives estimate values of all parameters used in the model. In our trial scheme, we aim at recovering 2 parameters: $\theta_1=\epsilon_0$ and $\theta_2=p$. We use $H( x)=\overline{H}( x)=\sum_{i\in I_{ x}} x_i$, $\Sigma_f=10$ uniformly for all classes. We run the genetic algorithm as a core program to solve both \eqref{eq:lsbb2} and \eqref{eq:mlebb2} with the following computer specification: operating system OSX 10.9.4, processor 2.6 GHz, RAM 16 GB, programming language Python 2.7.6 64bits and accuracy 4 digits. The corresponding results can be looked up in Table \ref{tab:GA}. 

The impact of the two important parameters in the model on the magnitude of the basic mosquito offspring number is shown in Fig.~\ref{fig:rp}. Concurrently, from Figs.~\ref{fig:u},~\ref{fig:uncont} and~\ref{fig:cont} we show the performance and attainment of optimal control to suppress the size of mosquito population.
\begin{table*}[htbp!]
  \centering
  \begin{tabular}{cccccccccccccc}
    \hline
    $\epsilon$ & $\epsilon_0$& $p$ & $q$ & $r$ & $s$ & $\beta_1$ & $\beta_2$ & $\beta_3$ & $\beta_4$ & $\gamma_1$ & $\gamma_2$ & $\mu_1$ & $\mu_2$\\
    \hline
    3 & 2 & 0.4 & 0.04 & 0.05 & 0.05 & 0.3 & 0.2 & 0.08 & 0.05 & 0.004 & 0.0026 & 0.02 & 0.01\\
    \hline
  \end{tabular}
\end{table*}
\begin{table*}[htbp!]
  \centering
  \begin{tabular}{ccccccccccccc}
    \hline
     $\mu_3$ & $\mu_4$ & $\mu_5$ & $T$ & $ h$ & $n$ & $\omega_{ x,\{1,2,3,4\}}$ & $\omega_{ x,5}$ & $\omega_{ u,\{1,2\}}$ & $a$ & $ x_0$ & $\sigma$\\
    \hline
    0.02 & 0.01 & 0.4 & 151 & [9,19] & 1 & 1 & 2 & $4\times10^3$ & [1,1] & [21,43,24,37,8] & $\frac{2\pi}{\lfloor T\slash 4\rfloor}$\\
    \hline
  \end{tabular}
  \caption{\label{tab:par}Estimated values of all parameters involved in the model \eqref{eq:sysbb}.}
\end{table*}
\begin{table*}[htbp!]
  \centering
  \begin{tabular}{c|c|ccccc}
    \hline
    Problem & & Scheme 1 & Scheme 2 & Scheme 3 & Scheme 4 & Scheme 5\\
    \hline
    & IT&100 & 1000 & 2000 & 5000 & 10000\\
    \hline
     \eqref{eq:lsbb2} & ET& 17 s & 125 s & 323 s & 901 s & 1987 s\\
     $\Delta t=0.1$ & $\theta_1$ & 1.7213 & 1.7326 & 1.8001 & 1.8995 & 1.9021\\
     & $\theta_2$& 0.2001 & 0.2567 & 0.3210 & 0.3334 & 0.3789\\
    \hline
     \eqref{eq:mlebb2} & ET& 191 s & 295 s & 692 s & 1443 s & 3403 s\\
     $\Delta t=0.1$ & $\theta_1$ & 1.777 & 1.8743 & 1.9123 & 1.9375 & 1.9786\\
     & $\theta_2$& 0.2932 & 0.3031 & 0.3635 & 0.3800 & 0.3994\\
    \hline
    \eqref{eq:lsbb2} & ET& 185 s & 1367 s & 3715 s & 10101 s & 27456 s\\
     $\Delta t=0.01$ & $\theta_1$ & 1.8514 & 1.9522 & 1.9812 & 1.9991 & 1.9999\\
     & $\theta_2$& 0.3465 & 0.3821 & 0.4187 & 0.4097 & 0.4091\\
    \hline
     \eqref{eq:mlebb2} & ET& 2921 s & 3229 s & 7942 s & 19534 s & 48047 s\\
     $\Delta t=0.01$ & $\theta_1$ & 1.9987 & 1.9998 & 1.9999 & 2.0000 & 2.0000\\
     & $\theta_2$& 0.3985 & 0.4023 & 0.4003 & 0.4000 & 0.4000\\
    \hline
  \end{tabular}
  \caption{\label{tab:GA}The standard genetic algorithm (GA) performance in solving \eqref{eq:lsbb2} and \eqref{eq:mlebb2}: IT = number of iterates, ET = elapsed time, s = second.}
\end{table*}

\begin{figure*}[h!]
\begin{minipage}[b]{.48\textwidth}
\centering
\includegraphics[scale=0.45]{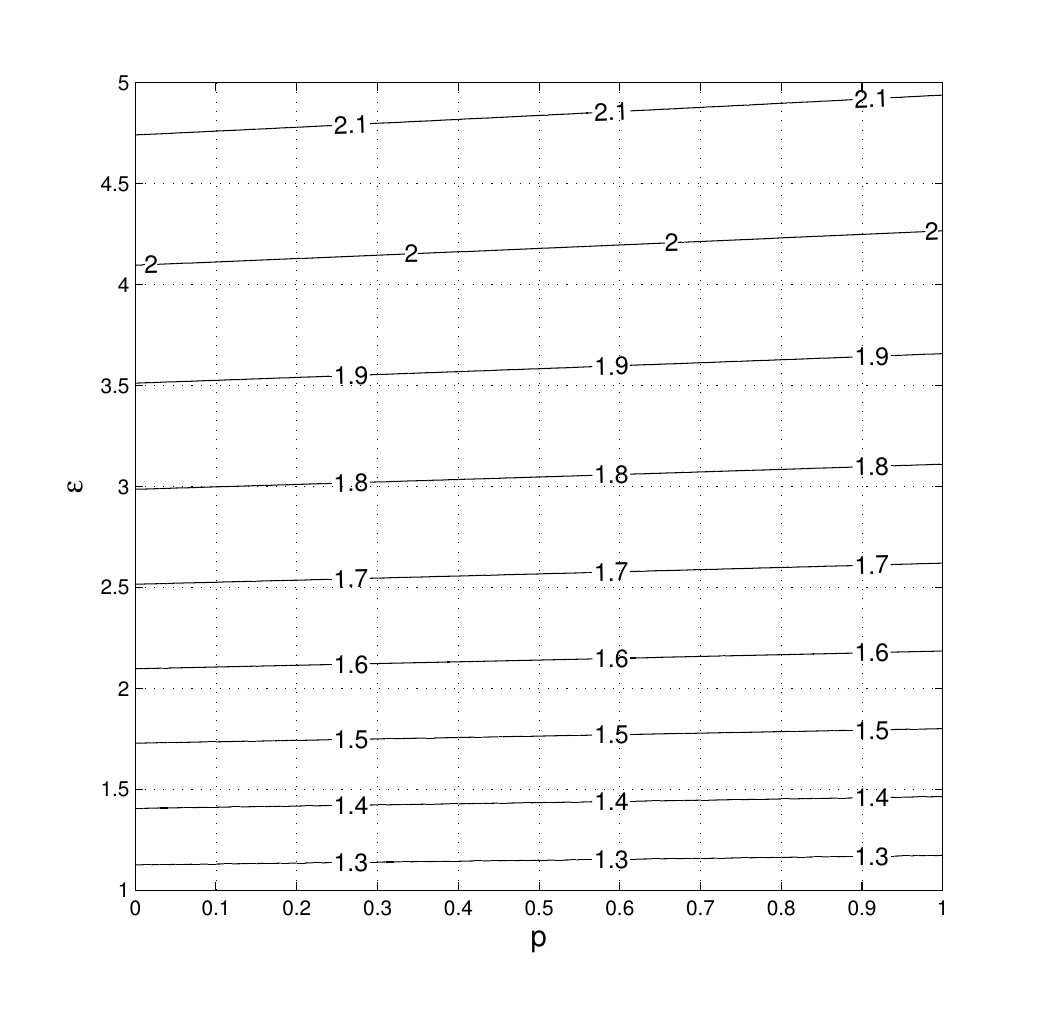}
\caption{The basic mosquito offspring number $\RM(d_3,d_4)$ in $(p,\epsilon)$-plane.}
\label{fig:rp}
\end{minipage}
\hfill
\begin{minipage}[b]{.48\textwidth}
\centering
\includegraphics[scale=0.45]{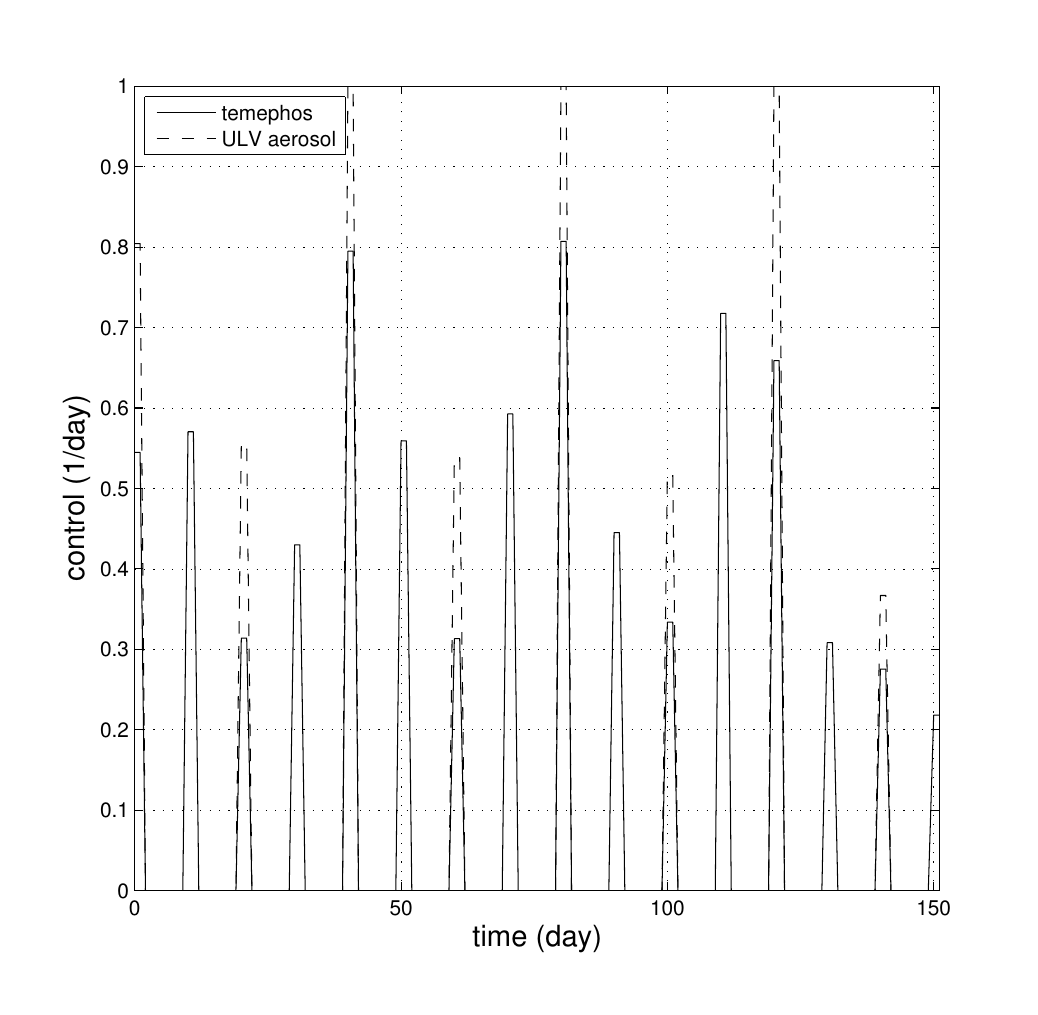}
\caption{Optimal control result.                                 ~~~~~~~~~~~~~~~~~~~~~~~~~~~~~~~~~~~~~~~~~~~~~~~~~~~~~~~~~~~~~~~~~~~~~~~~~~~~~~~~~~~~~~~~~~~~~~~~~~~~~~~~~~~~~~~~~~~~~~~~~~~~~~~~~~~~~~~~~~~~~~~~~~~~~~~~~~~~~~~~~~~~~~~~~~~~~~~~~~~~~~~~~~~~~~~~~~~~~~~~~~~~~~~~~~~~~~~~~~~~~~~~~}
\label{fig:u}
\end{minipage}
\end{figure*}

\begin{figure*}[h!]
\begin{minipage}[b]{.48\textwidth}
\centering
\includegraphics[scale=0.45]{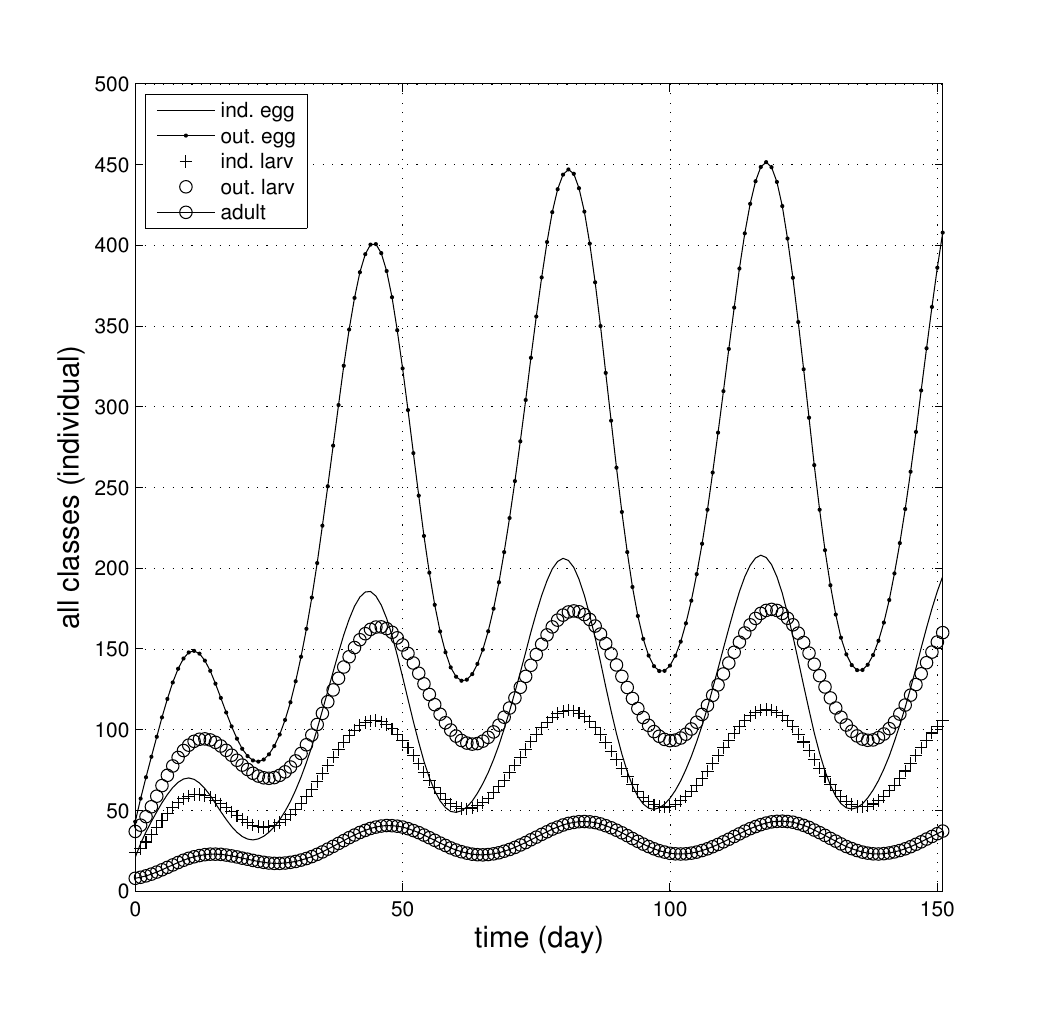}
\caption{Uncontrolled dynamics.}
\label{fig:uncont}
\end{minipage}
\hfill
\begin{minipage}[b]{.48\textwidth}
\centering
\includegraphics[scale=0.45]{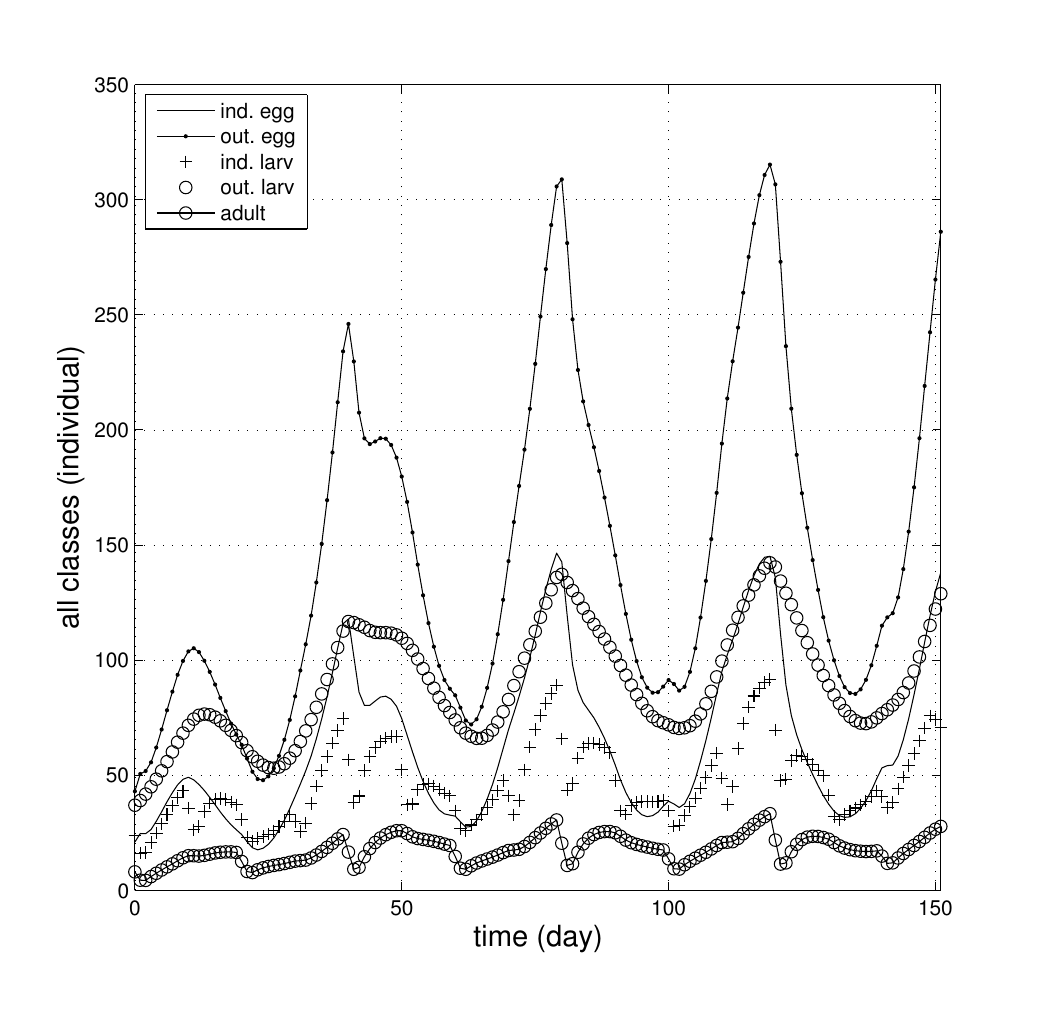}
\caption{Controlled dynamics.}
\label{fig:cont}
\end{minipage}
\end{figure*}
\section{Concluding remarks}
We exhibit a mosquito population dynamic model using some unifying theories bearing from non-autonomous dynamical system. Imposing relevant assumptions over all parameters in the model, we prove the positivity, uniqueness and boundedness of the corresponding solution. Using the \textsc{Floquet} theory, we prove that the trivial periodic solution exists if the basic mosquito offspring number $\RM(d_3,d_4)\neq 1$ and is asymptotically stable if $\RM(d_3,d_4)<1$. In this paper, we cannot derive a direct relationship between existence and stability of a nontrivial periodic solution $\nu$ that corresponds to the nontrivial autonomous equilibrium $ Q $ and the basic mosquito offspring number. Using alternative basic offspring numbers, it states whenever $\max\left\{\RM(d_3+2\gamma_1 x_3^{ \ast },d_4+2\gamma_2 x_4^{ \ast }),\RM(d_3+2\gamma_1\nu^{\min}_3,d_4+2\gamma_2\nu^{\min}_4)\right\}<1$ (see Lemma~\ref{lem:trivialbb} and Theorem~\ref{thm:nontrivialbb}), then $\nu$ exists and is asymptotically stable. A hypothetical result from Table~\ref{tab:par} and Fig.~\ref{fig:rp} suggests that any set of parameters satisfying $\RM(d_3,d_4)>1$ makes the corresponding solution in Fig.~\ref{fig:uncont} attracting a nontrivial periodic solution, therefore the periodic solution tends to graphically be asymptotically stable. From Fig.~\ref{fig:rp}, one has to reduce $\epsilon$ (ensure that meteorology does no longer support mosquito life) in order to reduce the magnitude of the basic mosquito offspring number. A condition when $\RM(d_3,d_3)<1$ guarantees that the mosquito population can return in an insignificant number within finite time and completely die out in expanded time.

Parameter estimation in our paper bears from necessity of suitable codes which have extreme reliability in real implementation. Three generic methods: Local Linearization (LL), \textsc{Pad\'{e}} approximation and Genetic Algorithm (GA) come into the play. Summarizing the performance of our codes (ref. Table~\ref{tab:GA}), one shows that the program executes in exponential time with the low convergence on average. Meanwhile, it is highlighted in the table that MLE scheme converges faster than LS scheme with respect to the number of iterates. Respective to naive implementation of GA, we examine that this low convergence results from the expensive evaluation of the objective function, even rigorous computation of LL solution on each iterate. We strive to work on finding efficient codes of derivative-use method as a striking resemblance with GA. This in turn enables us to compare both methods within the pursuit of the most efficient code applied to our framework.

We present some brief conclusions from application of optimal control. In this paper, our work is circumscribed by the application of constant collocation. The constant collocation practically means that we impose a constant control at a certain day of treatment. In an endemic area, certain amount of control has to be deployed uniform to the spatial and time (in 1 day) thematics. This can be more or less a key step towards development of efficient distribution of the control. Further application of polynomial collocation of degree $\geq 1$ is needed to propose well-suited program, being readable throughout academia. By comparing Figs.~\ref{fig:uncont} and \ref{fig:cont}, we conclude that the optimal control can generally reduce the size of the mosquito population. From Fig.~\ref{fig:u}, we note that the application of the ULV aerosol is preferred over that of temephos. It is also concluded that the pattern of optimal control fluctuates with the same tendency as that of the model solution.


\end{document}